\documentclass{lmcs}
\pdfoutput=1

\usepackage{lastpage}
\lmcsdoi{18}{3}{21}
\lmcsheading{}{\pageref{LastPage}}{}{}%
{Jan.~31,~2018}{Aug.~09,~2022}{}


\input{base}
\usepackage[inline]{enumitem}

\theoremstyle{definition}

\usetikzlibrary{decorations.pathreplacing}

\begin{document}

\title{The Third Trick}

\author[H.~Diener]{Hannes Diener}	%required
\address{School of Mathematics and Statistics, University of Canterbury, Christchurch, New Zealand}	%required
\email{\{\texttt{hannes.diener},\texttt{matthew.hendtlass}\}\texttt{@canterbury.ac.nz}}  %optional
%\thanks{thanks 1, optional.}	%optional

\author[M.~Hendtlass]{Matthew Hendtlass}	%optional
\address{School of Mathematics and Statistics, University of Canterbury, Christchurch, New Zealand}
\email{matthew.hendtlass@canterbury.ac.nz}  %optional

%% etc.

%% required for running head on odd and even pages, use suitable
%% abbreviations in case of long titles and many authors:

%%%%%%%%%%%%%%%%%%%%%%%%%%%%%%%%%%%%%%%%%%%%%%%%%%%%%%%%%%%%%%%%%%%%%%%%%%%

%% the abstract has to PRECEDE the command \maketitle:
%% be sure not to issue the \maketitle command twice!

\begin{abstract}
We prove a result, similar to the ones known as Ishihara's First and Second Trick, for sequences of functions. 
\end{abstract}

\maketitle

\section{Introduction}
In Bishop's constructive mathematics (BISH)\footnote{Informal mathematics using intuitionistic logic and an appropriate set-theoretic or type-theoretic foundation such as~\cite{pA01}. See~\cite{eB67,dB85} for details. We do assume dependent/countable choice, but will explicitly label any use thereof.}, one often has to navigate around the reality that one cannot make use of the law of excluded middle. Even though constructivists assume that we can make decision about \emph{finite} objects, case distinctions of the kind of 
\begin{equation} \label{PR:LPO}
\fa{x \in \RR}{x < 0 \lor x=0 \lor x>0}	
\end{equation}
are  unavailable. Most of these general disjunctions are actually false in constructive varieties such as Brouwer's intuitionism (INT) or Russian recursive mathematics (RUSS), while in BISH they are merely not acceptable since the latter is consistent with not just constructive varieties, but also classical mathematics (CLASS). One easily reaches the conclusion that there are no such disjunctions available in BISH. 

This is what makes the results nowadays known as Ishihara's First and Second Trick~\cite{hI91} so deliciously surprising: they allow us to make an interesting and non-trivial decision about ideal objects.
Both Tricks assume a strongly extensional\footnote{Defined below.} mapping \(f\) of a complete metric space \((X,\rho)\) into a metric space \((Y,\rho)\)\footnote{We will use $\rho$ as denoting the metric on any metric space we consider for simplicity. There is no point at which this leads to ambiguity.}, and a sequence \((x_{n})_{n \geqslant 1}\)   in \(X\) converging to a limit \(x\).
\begin{Proposition}[Ishihara's First Trick] For all positive reals \(\alpha < \beta\), 
 \begin{equation*} \label{Eqn:IshFirst}
 	\ex{n \in \NN}{\rho(f(x_{n}),f(x)) > \alpha} \  \lor \  \fa{n \in \NN}{\rho(f(x_{n}),f(x)) < \beta} \ .
 \end{equation*} 
\end{Proposition}
\begin{Proposition}[Ishihara's Second Trick] For all positive reals \(\alpha < \beta\), either we have 
\begin{itemize}
  \item \(\rho(f(x_{n}),f(x)) < \beta\) eventually,
  \item  or \(\rho(f(x_{n}),f(x)) > \alpha\) infinitely often.
\end{itemize}
\end{Proposition}

Even though it is wrong (see, however, Proposition~\ref{Pro:IshSecTrickForNinf}), it is \emph{helpful} to think of Ishihara's Second Trick as saying that a strongly extensional function is either sequentially continuous or it is discontinuous. 

In this note, we will show that we can prove a similar result for a sequence of functions converging point-wise to another function. 
Vaguely speaking we can show that one can decide whether convergence  happens somewhat uniform, or in a discontinuous fashion.

The motivating example for our results is the following. Consider the sequence of functions \( (f_n)_{n \geqslant 1} \) defined by 
\begin{equation*} 
f_n(x) = \begin{cases} 2n x & \text{if } x \in \lbrack 0,\frac{1}{2n} \rbrack \\ 1-2n x & \text{if } x \in \lbrack\frac{1}{2n}, \frac{2}{2n}\rbrack \\ 0  & \text{if } x \in \lbrack\frac{2}{2n},1\rbrack  \ . \end{cases} 
\end{equation*}

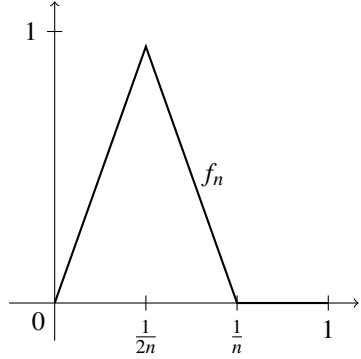
\begin{center}
\begin{tikzpicture}[y=-1cm]
\draw[->] (3,7) -- (3.6,7) node[anchor=north east] {$0$} -- (7.6,7);
\draw[->] (3.6,7.5) -- (3.6,3);
\draw (4.8,6.9)  -- (4.8,7.1) node[anchor=north] {$\frac{1}{2n}$};
\draw (6,6.9) -- (6,7.1) node[anchor=north] {$\frac{1}{n}$};
\draw (7.2,6.9) -- (7.2,7.1) node[anchor=north] {$1$};
\draw (3.5,3.4) node[anchor=east] {$1$} -- (3.7,3.4) ;
\draw[thick] (3.6,7) -- (4.8,3.6) -- (6,7) node[anchor=west, pos=0.5] {$f_n$} -- (7.2,7);
\end{tikzpicture}
\end{center}

The sequence \((f_n)_{n \geqslant 1}\) is one of the standard examples in classical analysis (see, for example, \cite[VII.2 Problem 2]{jD60}) of a sequence of functions on the unit interval that converges point-wise but not uniformly. However, constructively this example breaks, since the assumption that we have point-wise convergence implies the limited principle of omniscience (\LPO), which states that for every binary sequence \( (a_n)_{n \geqslant 1} \) we have \[ \fa{n \in \NN}{a_n =0} \lor \ex{n \in \NN}{a_n = 1} \ , \]
which is, under the assumption of countable choice, equivalent to Equation \ref{PR:LPO}.

\begin{Proposition}
	\LPO is equivalent to the statement that \((f_n)_{n \geqslant 1}\)  defined as above converges point-wise to \(0\).
\end{Proposition}
\begin{proof}
	Let \((a_n)_{n \geqslant 1}\) be a binary sequence. Without loss of generality we may assume that \((a_n)_{n \geqslant 1}\) is increasing. Now consider the sequence \((x_n)_{n \geqslant 1}\) in \( [0,1]\) defined by
 \[ 
   x_{n} = \begin{cases}  
               \frac{1}{2m} & \text{if } a_{n} = 1 \text{ and } a_{m}=1- a_{m+1}  \\ 
               0 & \text{if } a_{n} = 0  \ .
                                                    \end{cases}
 \]
Using the notation of~\cite{hD12b, hD17} this is simply the sequence \( \left((a_n) \circledast (\frac{1}{2n})\right) \). It is easy to see (or formally proven in~\cite[Lemma 2.1]{hD12b}), that \((x_n)_{n \geqslant 1}\) is a Cauchy sequence. It therefore converges to a limit \(z \in [0,1]\). Since \(f_n(z) \to 0\), there exists \( N \in \NN \) such that 
\begin{equation} \label{Eqn:MotivExm}
	\fa{n \geqslant N}{f_n(z) < \frac{1}{2} } \ .
\end{equation}
 Now either \(a_{N+1}=1\) or \(a_{N+1} = 0\). In the first case we are done. In the second case there cannot be \(m > N +1\) such that \(a_m =1\). For assume there is such an \(m\), then we can find \( N < m^\prime < m\) such that \(a_{m^\prime} = 0\) and \(a_{m^\prime+1} = 1\). In that case \(x_n\) is eventually constant on \(\frac{1}{2m^\prime}\), which means that \(z = \frac{1}{2m^\prime}\). However \( f_m(z) = 1\) and \( m \geqslant N\), which would be a contradiction to Equation \ref{Eqn:MotivExm}. Thus \(a_n =0\) for all \(n \in \NN\). Hence \LPO holds.
 
Conversely, it is easy to see that \LPO in the form of Equation \ref{PR:LPO} is enough to show that  \((f_n)_{n \geqslant 1}\) converges point-wise.
\end{proof}

That means that in varieties of BISH in which \LPO is false the above sequence of function actually does not provably converge point-wise; which means the above example cannot serve to show that uniform convergence is not implied by point-wise convergence. As it turns out there are (constructive) scenarios where the two notions coincide.

Looking at this from a different angle one could say that the assumption of point-wise convergence is constructively a stronger assumption than classically. 

Of course, working in BISH we only know that \LPO is not provable, but not whether \nLPO is provable. That is if, in BISH, we encounter a sequence of functions converging point-wise there are, intuitively speaking, two options: the convergence is actually uniform or it is not, in which case \LPO holds. 

This is not just an external disjunction: surprisingly we can make this decision within BISH and without knowing whether \LPO holds or fails.

\section{The Third Trick}
We will first consider a restricted version of our results applying only to functions of a specific type signature. In our opinion, this has the advantage of making the actual underlying ideas and structure of our proofs cleaner and clearer.

As Escardó has shown~\cite{mE13b,mE13}, the natural setting for Ishihara's tricks is  \( \Ninf \)---the space of all increasing binary sequences with the metric induced by the usual one on Cantor space. The space \(\Ninf\) contains the sequences \( \underline{n} = 0^n 1 \dots \) for any \(n \in \NN\), and the sequence \( \omega = 0 00  \dots \).  Using classical logic we have that \[ \Ninf = \set{\underline{n}}{n \in \NN} \cup \menge{\omega}  \ , \] 
however this cannot be proven with intuitionistic logic alone, since that statement is actually equivalent to \LPO.

We also assume that the set \( \Ninf \) is equipped with the reverse  lexicographic order \(\leqslant \); so, for example, \(\underline{n} \leqslant \underline{n+1} < \omega \).\footnote{Of course, this order is not decidable, constructively.}

For functions \(\Ninf \to \NN\) Ishihara's Second Trick becomes

\begin{Proposition}[Ishihara's Second Trick for \(\Ninf \to \NN\)] \label{Pro:IshSecTrickForNinf} $ $
\newline If $f: \Ninf \to \NN$ is strongly extensional, then 
\[\ex{N \in \NN}{ \fa{n \geqslant N}{f(\underline{n}) =  f(\omega)}} \ \lor \ \ex{\alpha \in \BS_{\textrm{inc}}}{\fa{n \in \NN}{f(\underline{\alpha(n)}) \neq  f(\omega)}} \ , \]
where $\BS_{\textrm{inc}}$ is the space of all strictly increasing functions $\BS$.
\end{Proposition}
Notice that for a function \(f\) of type \(\Ninf \to \NN\) this really states that we can decide whether \(f\) is continuous or find a witness of discontinuity. 

Similarly to Ishihara's Second Trick being an iteration of the first one, our Third Trick will also rely on iterating a simpler result, which is Part 2 of the following.

\begin{Lemma} \label{Lem:trick2.5}
Assume that \( (f_n)_{n \geqslant 1} : \Ninf \to \NN \) is such that \( f_n(x) \to 0 \) for all \(x \in \Ninf\).
\begin{enumerate}
  \item For any \( \gamma \in \Ninf \)
\begin{equation*}
	\ex{ i \in \NN }{f_i (\gamma) \neq 0} \:\: \lor \:\:  \fa{i \in \NN}{f_i (\gamma) = 0} \ .
\end{equation*}
\item Either \begin{itemize}
  \item there exists \(i \in \NN \) and \(\alpha \in \Ninf \) such that  \(f_i(\alpha) \neq 0 \)
  \item or, for all \(i \in \NN \) and \(\alpha \in \Ninf \), we have that  \(f_i(\alpha) = 0 \)
\end{itemize}	
\end{enumerate}
\end{Lemma}
\begin{proof}
\begin{enumerate}
  \item 
For any \( \gamma \in \Ninf \), since \(f_n (\gamma) \to 0 \), there exists \( N_\gamma\) such that \(f_i (\gamma) = 0\) for all \( i \geqslant N_\gamma \). Because we only need to check finitely many terms, either there exists \( i < N_\gamma  \) such that \( f_i(\gamma) \neq 0\) or not.
\item
Using the first part of this lemma we can, first exclude the case that there exists \(i \in \NN\) such that \(f_i(\omega) \neq 0 \), since that means we are done. So, for the rest of this proof we can assume that \( f_i(\omega) = 0 \) for all \( i \in \NN\). 

Also, using the first part of this lemma, we can build (using unique choice) a binary  sequence \( \lambda_k\) such that 
\begin{align*}
	\lambda_k = 0 &\implies  \fa{i}{f_i(\underline{k}) =0} \ , \\
	\lambda_k = 1 &\implies  \ex{i}{f_i(\underline{k}) \neq 0}	\ .
\end{align*}
Now define a sequence \(\beta_k \in \Ninf \) by 
\begin{align*}
	\beta_k = \underline{n} & \iff \fa{i < n}{\lambda_n = 0} \ , \\
	\beta_k = \beta_{k-1} & \iff \ex{i < n}{\lambda_i = 1} \ .
\end{align*} 
In words, going through the sequence \( \lambda \), as long as \(\lambda_k\) is \(0\) we just set \(\beta_k = \underline{k}\), but as soon as we hit a \(k\) with \(\lambda_k = 1\) we stay constant at that first \( \underline{k}\).
This is a Cauchy sequence, which therefore converges to a limit \(\beta \in \Ninf\).
Now, by Lemma \ref{Lem:trick2.5}.1 either there exists \( j \in \NN \) such that \(f_j(\beta) \neq 0\) or not. In the first case we are done, so let us focus, for the rest of the proof, on the second case. 

In this second case we claim that we must have \(\lambda_k = 0\) for all \(k \). For assume there is \(k \) such that \(\lambda_k = 1\). We may assume that \(k\) is the minimal such number. In that case we have \(\beta = \underline{k} \) and there exists \(i\) such that \(f_i(\beta) =f(\underline{k} )\neq 0\). But that is a contradiction to the case we are in, and therefore \( \lambda_k = 0\) for all \( k \in \NN\).
Now there cannot be a \( \alpha \in \Ninf \) and \( i \in \NN \) such that \( f_i(\alpha) \neq 0 \). For assume there is such \(\alpha\) and \(i\). Then \( \alpha = \omega \), since, if there exists \( n \) such that \( \alpha(n) = 1 \), we would have \( \lambda_{n} = 1 \). But that means that \(0 \neq f_i(\alpha) = f_i(\omega) \), which we excluded at the beginning of the proof.  \qedhere
\end{enumerate}
\end{proof}

\begin{Theorem}[The Third Trick for \(\Ninf \to \NN\)] \label{Thm:3rdTrickNinf} 
	If \( (f_n)_{n \geqslant 1} : \Ninf \to \NN \) such that \( f_n(x) \to 0 \) for all \(x \in \Ninf\), then either 
	\begin{itemize}
  \item there exists \(N \in \NN \) such that for all \( i \geqslant N \) and \( \alpha \geqslant \underline{N} \) we have \(f_i(\alpha) = 0 \), that is the convergence is uniform,
  \item or there exists a sequence \( \alpha_n\) in \( \Ninf \) and sequence \( k_n \) such that \(k_n \geqslant n \), \( \alpha_n \geqslant \underline{n}\), and \( f_{k_n}(\alpha_n) \neq 0  \), that is there is a witness showing that the convergence is not uniform.
\end{itemize}
\end{Theorem}
\begin{proof}
Since \(f_n(\omega) \to 0\) we may assume, without loss of generality, that \( f_n(\omega) = 0\) for all \(n \in \NN\). Using the second part of the previous lemma, fix a binary sequence \( \lambda_n \)  (using  countable choice to collect the \(\alpha_n\)) such that
\begin{align*}
\lambda_n = 0 & \implies \ex{i \geqslant n,\alpha_n \geqslant \underline{n}}{ f_i(\alpha_n) \neq 0} \ , \\
\lambda_n = 1 & \implies  \fa{i \geqslant n, \alpha \geqslant \underline{n}}{f_i(\alpha) = 0} \ .
\end{align*}
The sequence \(\lambda_n\) is increasing, and we may assume that \(\lambda_1 = 0\), since we are otherwise done.

Now define a sequence \((\beta_n)_{n \geqslant 1}\) in \(\Ninf\) by 
\begin{align*}
	\lambda_n = 0 &\implies \beta_n = \alpha_n  \ , \\
	\lambda_n = 1 &\implies \beta_n = \alpha_m, \text{ where } \lambda_m =0 \land \lambda_{m+1} =1 \ .
\end{align*}
In words, as long as \(\lambda_n =0\) we set \(\beta_n = \alpha_n \), and as soon as we find the first term such that \( \lambda_n =1\), we stay constant on the previous term \( \alpha_{n-1} \). This ensures that all \(\beta_n\) are such that there exists \(i\) with \(f_i(\beta_n) \neq 0\). 
Again, \((\beta_n)_{n \geqslant 1}\) is easily seen to be a Cauchy sequence, which therefore converges to a limit \(\gamma\) in \(\Ninf\). 

Since $f_n(\gamma) \to 0$ there is $N_\gamma$ such that $f_n(\gamma) = 0$ for all $n \geqslant N_\gamma$, which means that by checking all $i < N_\gamma$ we can decide whether  either $f_i(\gamma) =0$ for all \(i \in \NN\) or whether there is $k$ such that $f_k(\gamma) \neq 0$ but $f_j(\gamma)=0 $ for $j > k$.

In the first case we must have \(\lambda_n = 0 \) for all \(n \in \NN\): for assume there exists \(m \) such that  \(\lambda_m = 1\). That means we can find \(m^\prime < m\) such that \(\lambda_{m^\prime} =0\) and \(\lambda_{m^\prime+1} =1\). But that implies that \(\beta_n = \alpha_{m^\prime}\) for all \(n > m^\prime \), and therefore \(\gamma = \alpha_{m^\prime} \). Thus there exists \(i > m^\prime \) such that \(f_i(\gamma) \neq 0\); a contradiction to the case we are in. 

In the second case we must have $\lambda_{k^\prime} = 1$, since otherwise we can reach the following, two-step contradiction. If there is $k^\prime > k$ such that $\lambda_{k^\prime} = 0$ and $\lambda_{k^\prime+1} = 1$, then $\gamma = \alpha_{k^\prime}$ and $\ex{i \geqslant {k^\prime}}{f_i(\gamma) \neq 0}$, which contradicts our choice of $k$.
Thus $\lambda_n = 0$ for all $n \in \NN$, which means $\gamma = \omega$, but that is also a contradiction, since we assumed that $f_n(\omega) = 0$ for all $n \in \NN$.

Together we can  decide whether $\fa{n \in \NN}{\lambda_n =0}$ or whether $\ex{n \in \NN}{\lambda_n =1}$  which means we are done, by the definition of $(\lambda_n)_{n \geqslant 1}$.
\end{proof}

\begin{Proposition} \label{Pro:3rdTrickLPO}
In the second case of the previous proposition \LPO holds.
\end{Proposition}
\begin{proof}
Let \((a_n)_{n \geqslant 1}\) be a binary sequence. Without loss of generality we may assume that \((a_n)_{n \geqslant 1}\) is increasing. Now consider the sequence \((x_n)_{n \geqslant 1}\) in \( \Ninf \) defined by
 \[ 
   x_{n} = \begin{cases}  
               \alpha_m & \text{if } a_{n} = 1 \text{ and } a_{m}=1- a_{m+1}  \\ 
               0 & \text{if } a_{n} = 0  \ .
                                                    \end{cases}
 \]
 Using the notation of~\cite{hD12b, hD17} this is simply the sequence \((a_n) \circledast (\alpha_n) \). It is easy to see (or follows from Lemma 2.1 of~\cite{hD12b}), that \((x_n)_{n \geqslant 1}\) is a Cauchy sequence. It therefore converges to a limit \(z\). 
 Since \(f_n(z) \to 0\) we can find \(N_z\) such that 
\begin{equation} \label{Eqn:conv3}
\fa{n \geqslant N_z}{f_n(z) = 0} \ .
\end{equation}
Now either \( a_{N_z} = 1\) and we are done, or \( a_{N_z} = 0\). In this second case there cannot be \(n > N_z \) such that \(a_n = 1\): for assume there is such \(n\). Then we can find  \(N_z \leqslant m < n\) such that \(a_m = 1 - a_{m+1} \), which means that \(z = \alpha_m\). By assumption there exists \(j \geqslant m\) such that \(f_j(\alpha_m) \neq 0\), but that is a contradiction to Equation~\ref{Eqn:conv3}. Thus, if \(a_{N_z} = 0\), then \( \fa{n \in \NN}{a_n = 0}\). Altogether \LPO holds.
\end{proof}

\begin{Corollary}
Under the assumption of \( \nLPO \), if  \( (f_n)_{n \geqslant 1} : \Ninf \to \{0,1\} \) is such that \( f_n(x) \to 0 \) point-wise then \(f_n \to 0 \) uniformly.
\end{Corollary}

\section{The General Case}

The situation is a bit more intricate when we move to more general metric spaces. Notice that, for example, in the case of \(\Ninf \to \NN\) being sequentially continuous is equivalent to being point-wise continuous, and even equivalent to being uniformly continuous.

The following definitions mirror these different levels of continuity. 
\begin{Definition}
Let \((f_n)_{n \geqslant 1}\) be a sequence of functions. We say that \((f_n)_{n \geqslant 1}\) converges
\begin{enumerate}
  \item \define{sequentially semi-uniform at $x$}, if for all \(x_n \to x\) and all \(\varepsilon >0\) there exists $N \in \NN$ such that 
\[ \fa{n,i \geqslant N}{ \rho\left(f(x_n),f_i(x_n)\right) < \varepsilon } \ ; \]
%\footnote{This notion is very similar, but not identical, to the classical notion of \define{continuous convergence}~\cite{}; which is almost exactly defined as here, apart from $f(x)$ instead of $f(x_n)$.}
  \item \define{semi-uniform at $x$}, if for  all \(\varepsilon >0\) there exists $N \in \NN$ and $\delta>0$ such that for all $y \in B_{x}(\delta)$
\[ \fa{i \geqslant N}{ \rho \left(f(y),f_i(y) \right) < \varepsilon }  \ . \]
\end{enumerate}
It is easy to see that we can also combine $N$ and $\delta$ in the second definition by using $\delta = \nicefrac{1}{N}$.
\end{Definition}

Trivially we have  
\[  \text{ uniform  } \implies \text{ semi-uniform } \implies \text{ seq.\ semi-uniform } \implies \text{ point-wise } \ .  \]

In this section we will show that all three implications can be reversed, assuming certain principles and working on certain spaces. 

\begin{Theorem}[General metric space version of the Third Trick] \label{Thm:3rdTrickMS}
Consider a sequence of functions  \((f_n)_{n \geqslant 1} : X \to Y\) defined on a complete\footnote{We would like to mention that here and in the following one can replace completeness with the much weaker notion of complete enough~\cite{hD12b,hD17}.} metric space $X$ into an arbitrary metric space $Y$ converging point-wise to $f:X \to Y$; let $(x_n)_{n \geqslant 1}$ be a sequence in  $X$ converging to $x \in X$, and consider $\varepsilon >0$.  Either
\begin{itemize}
\item there exists \(N \in \NN \) such that 
\[ \fa{n,i \geqslant N}{ \rho\left(f(x_n),f_i(x_n)\right) < \varepsilon } \ ; \]
  \item or there exists a sequence \( z_n \to x \) and \( k_n \) such that \(k_n \geqslant n \),  and \[ \fa{n \in \NN}{\rho\left(f_{k_n}(z_n),f(z_n)\right) > \frac{\varepsilon}{4} } \ . \]
\end{itemize}
Furthermore, in case the second alternative holds, \LPO holds.
\end{Theorem}
The proof very much follows the lines of the proof of Lemma \ref{Lem:trick2.5}, Theorem \ref{Thm:3rdTrickNinf}, and Proposition \ref{Pro:3rdTrickLPO} and has been moved to an appendix.

\begin{Corollary}
Assuming \nLPO,	 if $f_n \to f$ point-wise then  $f_n \to f$ sequentially semi-uniform.
\end{Corollary}
%\begin{proof}
%The assumption of \nLPO rules out the second alternative of the previous Lemma, which means $f_n \to f$ sequentially semi-uniform.
%\end{proof}

To get from  sequential semi-uniform convergence to semi-uniform convergence we can use Ishihara's boundedness principle \BDN, which states that every countable pseudo-bounded subset of $\NN$ is bounded.
Here, a subset $S$ of $\NN$ is \define{pseudo-bounded} if $\lim_{n \to \infty} s_n / n = 0$ for each sequence $(s_n)_{n \geqslant 1}$ in $S$.In~\cite[Lemma 3]{hI02}, it was shown that a set $S$ of natural numbers is pseudo-bounded if and only if for each sequence $(s_n)_{n \geqslant 1}$ in $S$, $s_n < n$ for all sufficiently large $n$. Every bounded subset of $\NN$ is trivially pseudo-bounded and, conversely, every inhabited, \emph{decidable}, and pseudo-bounded subset of $\NN$ is easily seen to be bounded. However, in the absence of decidability, this is not guaranteed anymore.

Overall \BDN is a very weak principle that is true in CLASS but which also holds in INT and RUSS. Indeed, there are very few known models in which \BDN fails. The first such model was a realizability model described in~\cite{pL04} and the second one was a topological model~\cite{rL12}.

\begin{Proposition} \label{Pro:12}
\BDN implies that for  $(f_n)_{n \geqslant 1}$ and $f$ defined on a countable space $X$ such that $f_n \to f$ sequentially semi-uniform, also converges semi-uniformly. 
\end{Proposition}
\begin{proof}
	Assume \(f_n \to f\) sequentially semi-uniform. Let \(X = \left\{ r_i \right\} \), and consider \(\varepsilon > 0\), and \(x \in X\). Consider the set
	\[ S = \set{n \in \NN}{\ex{i,j \geqslant n}{r_i \in B_{x}(\frac{1}{n}) \  \land  \ \rho(f(r_i),f_j(r_i)) > \varepsilon}} \cup \menge{0} \  .\] 
	This set is easily seen to be countable.  We will show that \(S\) is pseudo-bounded. To this end let \((s_n)_{n \geqslant 1}\) be a sequence in \(S\). Define a sequence \(x_n\) by 
	\begin{align*}
		s_n \leqslant n  & \implies x_n = x \\
		s_n > n &\implies x_n = r_i, \text{ where } i \text{ is as in the definition of } S  \ . 
	\end{align*}
Since, for all $n$, we have $\rho(x,x_n) < \frac{1}{n}$, this sequence converges to \(x\). Since $(f_n)_{n \geqslant 1}$, by assumption, converges sequentially semi-uniform there exists $N$ such that $\fa{i,n \geqslant N}{\rho(f(x_n),f_i(x_n)) < \varepsilon} $. That means, that for all $n \geqslant N$ we must have $s_n \leqslant n$, since otherwise $x_n$ would be such that there is $j \geqslant N$ with $\rho(f(x_n),f_j(x_n) > \varepsilon$, by the definition of $S$. Thus, by \BDN, the set $S$ is bounded. That means that for  all \(\varepsilon >0\) there exists $N \in \NN$ such that for all $y \in B_{x}(\frac{1}{N}) \cap \set{r_i}{i \in \NN}$
\[ \fa{i \geqslant N}{ \rho \left(f(y),f_i(y) \right) < \varepsilon }  \ , \] 
and since $\varepsilon$ and $x$ were arbitrary we are done.
\end{proof}

\begin{Corollary} \label{Cor:13}
\BDN implies that for every sequence of non-discontinuous functions $(f_n)_{n \geqslant 1}$ and $f$ defined on a separable space $X$ such that $f_n \to f$ sequentially semi-uniform, also converges semi-uniformly. 	
\end{Corollary}

For the next corollary we remind the reader that the existence of a discontinuous function is equivalent to the weak limited principle of omniscience \WLPO, and hence the negation of the latter implies that all functions defined on a \emph{complete} metric space are non-discontinuous~\cite[Theorem 1]{Ishihara1992}. 

\begin{Corollary} \label{Cor:14}
Assuming \nWLPO and \BDN (both hold in RUSS and INT). For a sequence of functions $f_n, f:X \to Y$ defined on a complete, separable $X$ and into an arbitrary metric space $Y$, point-wise and  semi-uniform convergence are equivalent.
\end{Corollary}

Finally, in order to also make the step to uniform convergence, we need to assume a form of Brouwer's fan theorem:  \FANP.
All versions of Brouwer's fan theorem  them enable one to conclude that a  bar is uniform. Here, a bar $B$ is a subset of the space of all finite binary sequences $2^\ast$ such that for all infinite binary sequences $\alpha$ there is $n$ such that $\overline{\alpha}n$---the initial segment of $\alpha$ of length $n$---is in $B$. A bar is uniform if this happens uniformly for all $\alpha \in \CS$, that is if
\[ \ex{N}{\fa{\alpha \in \CS}{\ex{n \leqslant N}{ \overline{\alpha}n \in B }}} \ .\]

A bar $B$ is called ${\Pi}_{1}^0$-bar, if there exist a set $S \subset \cS \times \NN$ such that 
\[ u \in B \iff \fa{n \in \NN}{ (u,n) \in S } \ , \]
and \[ (u,n) \in S \implies ( u \ast 0,n) \in S \land ( u \ast 1,n) \in S \ . \]
\FANP holds in INT as well as in CLASS and is slightly stronger than the uniform continuity theorem (\UCT), which is the statement that all point-wise continuous functions $\CS \to \RR$ are  uniformly continuous.

\begin{Proposition} \label{Pro:FANc_impl_semi_u_is_u}
\FANP implies that every sequence of non-discontinuous functions $f_n,f:  \CS \to \RR$ such that $f_n \to f$ semi-uniformly, also converges uniformly. 
	
	Conversely, the latter statement implies \UCT.
\end{Proposition}
\begin{proof}
	Let $f_n : \CS \to \RR$ be such that $f_n \to f$ semi-uniform. 
	Fix $\lambda_{u,n}$ such that 
	\begin{align*}
		\lambda_{u,n} = 0 & \implies \rho\left(f_n(u \ast 000\dots ),f(u \ast 000\dots)\right) < \varepsilon \ , \\
		\lambda_{u,n} = 1 & \implies \rho\left(f_n(u \ast 000\dots ),f(u \ast 000\dots)\right) > \nicefrac{\varepsilon}{2} \ .
	\end{align*}
Define a decidable set $S \subset \cS \times \NN$ by
\begin{equation} \label{Eqn:Pi01FAN}
	(u,n) \in S \iff  \fa{\abs*{u}\leqslant i \leqslant \abs*{u}+n }{\fa{w \in \cS}{ \abs*{w} \leqslant n - \abs*{u}  \implies \lambda_{u\ast w, i} = 0}}  \ .
\end{equation}
To obtain a $\Pi_1^0$-set define $B$  by $u \in B  \iff \fa{n \in \NN}{(u,n) \in S}$. 
We claim that $B$ is a $\Pi_1^0$-bar. Firstly, notice that the condition ``$\abs*{w} \leqslant n - \abs*{u} $'' easily ensures that if $(u,n) \in S$ then also $(u \ast i,n) \in S$ for $i=0,1$, as required. To see that $B$ is a bar let $\alpha \in \CS$ be arbitrary. By the assumption of semi-uniformity there exists $N$ such that for all $\beta \in \CS$ and all $n \geqslant N$ we have
\[ \rho \left( f_n(\overline{\alpha}N \ast \beta ), f(\overline{\alpha}N \ast \beta ) \right) < \nicefrac{\varepsilon}{2}  \ . \]
In particular we have that 
\begin{equation} \label{Eqn:CSclosure}
\fa{w \in \cS}{\fa{i \geqslant N}{ \lambda_{\overline{\alpha}N \ast w , i } = 0  \ . }}
\end{equation}
That ensures that $\overline{\alpha}N \in B$, since it actually over-fullfills  \ref{Eqn:Pi01FAN}.

Applying \FANP we get a uniform bound for  $B$; that is, there is $M$ such that $\overline{\alpha}M \in B$ for all $\alpha \in \CS$. Now let $v \in \cS$ and $n \in \NN$ such that both $\abs*{v},n \geqslant M$. Let $w$ be the suffix of  $v$: $v = \overline{v}M \ast w$. 

Let $k = \max \{  n-M, |w|+M \}$. Then   $n \leqslant M+k $, and $\abs*{w} \leqslant k -M $. That means that, since $(\overline{v}M, k ) \in S $, we have $\lambda_{v,n} = 0$, and hence, by the definition of $\lambda$ we have \[ \rho\left(f_n(v \ast 000\dots ),f(v \ast 000\dots)\right) < \varepsilon \] for $\abs*{v} \geqslant M$ and $n \geqslant M$.
Since we can pad out any $v \in \cS$ if it is shorter than $M$ by $0$s that actually means that for \emph{any} $v \in \cS$ and all $n \geqslant M$ we have  $\rho\left(f_n(v \ast 000\dots ),f(v \ast 000\dots)\right) \leqslant \varepsilon$.
Finally, by the non-discontinuity this means that for any $\alpha \in \CS$ and $n \geqslant M$ we have  $\rho\left(f_n(\alpha),f(\alpha)\right) < \varepsilon$. Hence the convergence is uniform.

Conversely to see that \UCT holds, it suffices to show that every point-wise continuous $f:\CS \to \RR$ is bounded~\cite[Theorem 10]{dB07}. We may assume that $f \geqslant 0$. Let $f_n = \min\{f,n\}$. Then $f_n \to f$ semi-uniformly, since a point-wise continuous function is locally bounded. Now if $f_n \to f$ uniformly then there exists $M \in \NN$ such that $\abs*{f_M(\alpha)-f(\alpha)} < 1$ for all $\alpha \in CS$. That means that there cannot be $\alpha$ such that $f(\alpha) > M+1$, since in that case $f(\alpha) > M+1 = f_M(\alpha) +1$, which means that $\abs*{f_M(\alpha)-f(\alpha)} > 1$. So $f(\alpha) \leqslant M+1$ for all $\alpha \in \CS$ and we are done.
\end{proof}

\begin{Corollary}
Assuming \FANP. For a sequence of non-discontinuous functions $f_n, f:X \to Y$ defined on a compact\footnote{As common in constructive analysis we define compact as complete and totally bounded.} $X$ and into an arbitrary metric space $Y$, semi-uniform and  uniform convergence are equivalent.
\end{Corollary}
\begin{proof}
By~\cite[Proposition 7.4.3]{Troelstra1988} there exists a uniformly continuous and surjective function $F:\CS \to X$.
Now let $(g_n)_{n \geqslant 1}: \CS  \to \RR$ be defined by \[ g_n(\alpha) =  \rho(f_n(F(\alpha)), f(F(\alpha)) ) \ . \] It is easy to see that $g_n$ is non-discontinuous and that $g_n \to 0$ semi-uniform. So by Proposition \ref{Pro:FANc_impl_semi_u_is_u} $g_n \to 0$ uniformly, which means for an arbitrary $\varepsilon > 0$ there exists $N \in \NN$ such that $\abs*{g_n} < \varepsilon$ for all $n \geqslant N$. Since $F$ is surjective,  for all $x \in X$ and $n \geqslant N$ we have $g_n(\alpha) =  \rho(f_n(x), f(x) )$. Hence $f_n \to f$ uniformly. 
\end{proof}

%\todo{Lemma not needed?}
%\begin{Lemma}
%The class of non-discontinuous functions is closed under addition and multiplication.
%\end{Lemma}
%\begin{proof}
%Consider $f,g: X \to Y$ non-discontinuous. Assume $f+g$ is discontinuous, so we have $x_n \to x$ and $\varepsilon > 0$ such that $\rho\left((f+g)(x_n), (f+g)(x) \right) > \varepsilon$. Also, \WLPO holds. Using countable choice we can fix a binary sequence $(\lambda_n)_{n \geqslant 1}$ such that
%\begin{align*}
%	\lambda_n=0 \implies \rho\left(f(x_n), f(x) \right) > \frac{\varepsilon}{2} \ , \\
%	\lambda_n=1 \implies \rho\left(g(x_n), g(x) \right) > \frac{\varepsilon}{2} 	\ .
%\end{align*}
%Using \WLPO iteratively we can decide whether either the first one is the case infinitely often, or the 
%\todo{details}
%
%\end{proof}

Since every compact space is, by definition, totally bounded, which in turns implies separability, we get the following, final corollary.

\begin{Corollary}
	Assuming \nWLPO, \BDN, and \FANP (all these hold in INT). For a sequence of functions $f_n, f:X \to Y$ defined on a compact $X$ and into an arbitrary metric space $Y$, point-wise and  uniform convergence are equivalent. 
\end{Corollary}

This  generalises the result of~\cite[\S 1 and \S5]{swart76a}, which there is proven with the help of continuous choice and the full fan theorem.

\section{Conclusion --- from the Second to the Third Trick?}

We should concede that, while our results are related to Ishihara's First and Second Trick, naming it the ``Third Trick'' might be slightly misleading. As mentioned above, the Second Trick is an iteration of the First Trick, so one might expect that the Third Trick is an iteration ---~or at least follows~--- from the Second one. This is not the case. One is tempted, in the situation that $f_n \to f$, where $f_n,f : \BS \to \NN$, to consider 
\[ F(\alpha) = \sum_{n \geqslant 0} \abs*{f_n(\alpha) - f(\alpha)} \ . \]
This is a well-defined function, since the sum is finite, because of the point-wise convergence. If we could apply Ishihara's Second Trick to $F$ we would immediately get (our main) Theorem \ref{Thm:3rdTrickNinf}. However, Ishihara's Second Trick requires strong extensionality, and that is not guaranteed with $F$. In~\cite{hD12b} the first author has given a version of Ishihara's second trick, which does not rely on strong extensionality, but which in turn has weaker consequences, which are not strong enough to deduce Theorem \ref{Thm:3rdTrickNinf}.

\section*{Acknowledgements}
The authors would like to thank the anonymous referee for pointing out the unnecessary assumption of completeness in Proposition \ref{Pro:12} and Corollary \ref{Cor:13} together with, ironically, the lack thereof in Corollary \ref{Cor:14}.

\bibliographystyle{alphaurl}
\bibliography{All}

\section*{Appendix: Proof of Theorem \ref{Thm:3rdTrickMS}}
\begin{proof}
The proof is analogous to Lemma \ref{Lem:trick2.5}, Theorem \ref{Thm:3rdTrickNinf}, and Proposition \ref{Pro:3rdTrickLPO}, which correspond to Part (b), (c), and (d) respectively. The only real added effort is in Part (b), where the loss of decidability is compensated by using approximate decisions for  $\frac{\varepsilon}{4},\frac{\varepsilon}{2}$, and $\varepsilon$.
\begin{enumerate}[(a)]
  \item  Let $x_n \to x$, and $\varepsilon > 0$. First, we may assume, without loss of generality, that $\fa{n \in \NN}{\rho(f_n(x),f(x)) < \frac{\varepsilon}{2}} $, since we can otherwise consider an appropriate tail of the sequence $(f_n)_{n \geqslant 1}$. Secondly, we fix a modulus of convergence $\mu \in \BS$ such that $ \rho(x_n ,x) < \frac{1}{k}$ for all $n \geqslant \mu(k)$. 
\item
Our first claim is that for any $n \in \NN$ we can decide whether there is $z_n$ and $i \in \NN$ such that $\rho(z_n,x) \leqslant \frac{1}{n}$ and \[ \rho(f_i(z_n), f(z_n)) > \frac{\varepsilon}{4} \] or whether $\rho(f_i(x_m), f(x_m)) < \varepsilon$ for all $i \geqslant n$ and $m \geqslant \mu(n)$. 

First, notice that for fixed $m$ there exists $N_m$ such that $\rho(f_i(x_m),f(x_m)) < \varepsilon$ for all $i \geqslant N_m$.  Thus we can decide whether there exists $i \geqslant n$ such that $\rho(f_i(x_m),f(x_m)) > \frac{\varepsilon}{2}$ or whether $\rho(f_i(x_m),f(x_m)) < \varepsilon$ for all $i \geqslant n$. Choose a binary sequence $\gamma_m$ that flags these possibilities by $1$ and $0$ respectively. 

Now define a sequence $(w_m)_{m \geqslant \mu(n)}$ such that $w_m = x$, if $\gamma_m = 0$, and if $\gamma_m = 1$ set $w_m=x_{m^\prime}$ where $ m^\prime$ is the smallest index such that  $\gamma_m =1 $. In words, as long as $\gamma_m = 0$ we stay constant on $x$ and if we ever hit a term such that $\gamma_m = 1$, we switch to $x_m$ and stay constant from then on. The sequence $w_m$ is easily seen to be a Cauchy sequence, and therefore converges to a limit $w$. Now, choose $N_w$ such that   $\rho(f_i(w),f(w)) < \frac{\varepsilon}{2}$ for all $i \geqslant N_w$. Since we only need to check for the indices between $n$ and $N_w$ (if any), we can decide whether either $\rho(f_i(w),f(w)) < \frac{\varepsilon}{2}$ for all $i \geqslant n$ or whether there is $j \geqslant n$ such that $ \rho(f_i(w),f(w)) > \frac{\varepsilon}{4}$. In the latter case we are done, since $ \rho(w,x) \leqslant \frac{1}{n} $, so we can choose $z_n = w$. In the first case we must have $\gamma_m = 0$ for all $m \geqslant \mu(n)$: for assume there is $\gamma_m =1$. Then we can find the first such index $m^\prime$, which means that $w = x_{m^\prime}$ and there exists $i \geqslant n$ such that $\rho(f_i(x_{m^\prime}),f(x_{m^\prime})) > \frac{\varepsilon}{2}$. This is a contradiction to $\rho(f_i(w),f(w)) < \frac{\varepsilon}{2}$ for all $i \geqslant n$.

\item 

Using the previous part, fix a binary sequence \( (\lambda_n)_{n \geqslant 1} \)  such that
\begin{align*}
\lambda_n = 0 & \implies \ex{i \geqslant n, z_n \in X }{\rho(x,z_n) \leqslant \frac{1}{n} \land \rho(f_i(z_n), f(z_n)) > \frac{\varepsilon}{4}}  \\
\lambda_n = 1 & \implies  \fa{i \geqslant n, m \geqslant \mu(n)}{\rho(f_i(x_m), f(x_m)) < \varepsilon} \ .
\end{align*}
We may assume that the sequence \(\lambda_n\) is increasing, and that \(\lambda_1 = 0\), since we are otherwise done.

Now define a sequence \((y_n)_{n \geqslant 1}\) in \(X\) by 
\begin{align*}
	\lambda_n = 0 &\implies y_n = z_n  \ , \\
	\lambda_n = 1 &\implies y_n = z_m, \text{ where } \lambda_m =0 \land \lambda_{m+1} =1 \ .
\end{align*}
In words, as long as \(\lambda_n =0\) we set \(y_n = z_n \), and as soon as we find the first term such that \( \lambda_n =1\), we stay constant on the previous term \( z_{n-1} \). Again, \((y_n)_{n \geqslant 1}\) is easily seen to be a Cauchy sequence, which therefore converges to a limit \(y\) in \(X\). 

Since $f_n(y) \to f(y)$ there is $N_y$ such that
\begin{equation} \label{Eqn:conve}
\rho(f_n(y),f(y)) < \frac{\varepsilon}{4} \text{ for all } n \geqslant N_y  \ . 
 \end{equation}

 Now either $\lambda_{N_y} = 1$ or $\lambda_{N_y} = 0$.

In the second case there cannot be $n > N_y$ with $\lambda_n = 1$: In that case we could find $N_y \leqslant n^\prime < n$ such that $\lambda_{n^\prime} = 0$ and $\lambda_{n^\prime+1} = 1$. But then $y = z_{n^\prime}$ and there is $i \geqslant n^\prime$ such that $\rho(f_i(y),f(y)) > \frac{\varepsilon}{4}$, which is a contradiction to Equation \ref{Eqn:conve}, since $i \geqslant n^\prime \geqslant N_y$. Hence $\lambda_n = 0$ for all $n > N_y$, which means that $\lambda_n = 0$ for all $n$, since $(\lambda_n)_{n \geqslant 1}$ is increasing.

Together we can  decide whether $\fa{n \in \NN}{\lambda_n =0}$ or whether $\ex{n \in \NN}{\lambda_n =1}$  which means---by the definition of $(\lambda_n)_{n \geqslant 1}$---we have proven the main claim of the theorem. 
\item To prove the ``\LPO claim'' of this theorem let \((a_n)_{n \geqslant 1}\) be a binary sequence. Without loss of generality we may assume that \((a_n)_{n \geqslant 1}\) is increasing. Now consider the sequence \((v_n)_{n \geqslant 1}\) in \( X \) defined by
 \[ 
   v_{n} = \begin{cases}  
               z_m & \text{if } a_{n} = 1 \text{ and } a_{m}=1- a_{m+1}  \\ 
               x & \text{if } a_{n} = 0  \ .
                                                    \end{cases}
 \]
Again, this is a Cauchy sequence, which therefore converges to a limit \(v\). 
 Since \(f_n(v) \to 0\) we can find \(N_v\) such that 
\begin{equation} \label{Eqn:conv4}
\fa{n \geqslant N_v}{\rho(f_n(v), f(v) ) < \frac{\varepsilon}{4}} \ .
\end{equation}
Now either \( a_{N_v} = 1\) and we are done, or \( a_{N_v} = 0\). In this second case there cannot be \(n > N_v \) such that \(a_n = 1\): for assume there is such \(n\). Then we can find  \(N_v \leqslant m < n\) such that \(a_m = 1 - a_{m+1} \), which means that \(v = z_m\). Also, by assumption there exists \(j \geqslant m\) such that \(\rho(f_j(z_m), f(z_m)) > \frac{\varepsilon}{4} \), but that is a contradiction to Equation~\ref{Eqn:conv4}. Thus, if \(a_{N_v} = 0\), then \( \fa{n \in \NN}{a_n = 0}\). Altogether \LPO holds.
 \qedhere
\end{enumerate}

\end{proof}

\end{document}